\documentclass[reqno]{amsart}
\usepackage[english]{babel}
\usepackage[latin1]{inputenc}
\usepackage{amsfonts}
\usepackage{amssymb}
\usepackage{amsmath}
\usepackage{amsthm}
\usepackage{graphicx}
\usepackage{epsfig}
\usepackage{fancyhdr}
\usepackage[all]{xy}

\usepackage{hyperref}
\usepackage{amsmath,amssymb}
\usepackage[bbgreekl]{mathbbol}

\numberwithin{equation}{section}

\newcommand{\C}{\mathbb{C}}

\newtheorem{theorem}{Theorem}[section]
\newtheorem{lemma}[theorem]{Lemma}

\newtheorem{defi}[theorem]{Definition}

\usepackage{amsmath}
\title [ H-integrable solutions to 2d dimensional canonical systems ]{H-Integrable Solutions of 2d-Dimensional Canonical Systems}
\author{Keshav Acharya}
\address{Embry-Riddle Aeronautical University \\ Daytona Beach, FL 32114, U. S. A.}
\date{March 2025}

\begin{document}

\maketitle

\begin{abstract} In this paper we study $2d$-dimensional canonical systems with positive semidefinite Hamiltonians $H(t)$ satisfying   $\operatorname{tr} H(t) \equiv 1$. We first establish that every such system possesses at least one nontrivial $H$-integrable solution, ensuring the existence of solutions whose energy with respect to $H$ remains finite. We then show that the space of all $H$-integrable solutions has dimension at most $d$.

Under a natural additional structural condition on $H(t)$, we prove that 
this upper bound is optimal. In this case, the system admits exactly $d$ 
linearly independent $H$-integrable solutions. Our approach combines symplectic geometry with the analysis of isotropic  subspaces.These results clarifies the   limit-point and limit-circle behavior associated with   higher-dimensional canonical systems.\\

  \end{abstract}

\section{Introduction}

Canonical systems play a central role in modern spectral theory and mathematical physics, serving as a unifying framework for a wide class of differential and operator equations. Originating from the work of de Branges \cite{deBranges1968}, these systems provide a powerful setting for studying  the spectral theory of several differential operators. In their general form, they are system of  first-order differential equation, expressed as  
\begin{equation} \label{ca}
J\,u'(t) = z\,H(t)\,u(t), \qquad t\in I\subset\mathbb{R},
\end{equation}
where $
J = 
\begin{pmatrix}
0 & I_d\\
-I_d & 0
\end{pmatrix},
$
\(u(t,z)\in\mathbb{C}^{2N}\) is a vector function, \(z\in\mathbb{C}\) is the spectral parameter, and \(H(t)\) is a locally integrable \(2d\times 2d\) a positive semidefinite matrix called the \emph{Hamiltonian}.  In the 2 dimensional case (\(d = 1\)), they unify Sturm--Liouville, Dirac, and Krein-string equations under a single first-order symplectic framework and  they have been well studied, see \cite{KRA, HSW, HSW1, HSW2, math1, Sakhnovich2010, Win} and the references therein. Extending the theory of these systems to higher dimension \(2d\) introduces rich geometric and analytic structures relevant to multi-component quantum systems, integrable models, and Hamiltonian dynamics~\cite{KLudu}. A key simplification arises from the \emph{trace-normalization} (or traced-normed) condition
\begin{equation}
\operatorname{tr}H(t)\equiv 1, \quad \text{for a.e. }t\in I.
\end{equation}
Any canonical system with positive-semidefinite \(H(t)\) can be transformed into trace-normed form by the reparameterization
\[
x(t) = \int_{t_0}^{t} \operatorname{tr}H(s)\,ds, \qquad 
\widetilde H(t) = [\operatorname{tr}H(t(x))]^{-1} H(t(x)).
\]
This normalization ensures that \(\|H(t)\|\le 1\), yielding uniform bounds on the coefficients and simplifying the spectral analysis. 
The trace-normed condition also provides a natural Hilbert-space structure: for any vector function \(u\), the \(H\)-inner product
\[
\langle u,v\rangle_H = \int_I u(t)^*\,H(t)\,v(t)\,dt,
\]
induces the weighted space \(L^2_H(I,\mathbb{C}^{2N})\).
Solutions \(u(t,z)\) satisfying
\[
\int_I u(t,z)^*H(t)u(t,z)\,dt < \infty,\]
are called \emph{\(H\)-integrable solutions}, and play a central role in defining the Weyl-Titchmarsh (M-)function, self-adjoint extensions, and spectral measures of the system.

The recent development of computational and analytical tools for \(2d\)-dimensional trace-normed canonical systems~\cite{KLudu, AL} builds directly upon the discrete groundwork laid in \cite{FischerRemling2009}. 
These systems with trace-normed and bounded Hamiltonian  opens research avenues in numerical spectral approximation, inverse problems, and stability analysis of multi-component physical models.

\section{Preliminaries and Results}

Let $ \mathbf I \subset \mathbb{R}$ be a bounded interval and $z \in \mathbb{C}$. If the Hamiltonian $H(t)$ is trace-normed and positive semidefinite, then there exists a unique vector-valued solution of the canonical system \eqref{ca} that is absolutely continuous almost everywhere on $\mathbf I,$ and satisfies the initial condition $u(t_0) = u_0$ for any fixed $u_0 \in \mathbb{R}^{2d}$, as shown in \cite{AL}. Consequently, for each $z \in \mathbb{C}$, the set of all solutions forms a vector space of dimension $2d$. By choosing $2d$ linearly independent solutions and arranging them as columns, we construct a $2d \times 2d$ matrix-valued function $Y(t, z)$, called the fundamental matrix solution of \eqref{ca}. Note that for any initial vector $v\in\mathbb C^{2d}$, the function $y(t)= Y(t;z)v$ is the unique absolutely continuous solution of \eqref{ca} with $y(0)=v$. These fundamental solutions obey some important properties. We begin by stating the following lemma from \cite{coddington-levinson}.

 \begin{lemma}[Gronwall - integral form]
Let $a<b$. Suppose $u:[a,b]\to[0,\infty)$ is continuous, $\alpha\in L^1([a,b])$ with
$\alpha\ge0$ a.e., and $v:[a,b]\to\mathbb{R}$ is absolutely continuous. If
\[
u(t)\le v(t)+\int_a^t \alpha(s)\,u(s)\,ds\qquad (t\in[a,b]),
\]
then
\[
u(t)\le v(t)\exp\!\Big(\int_a^t \alpha(s)\,ds\Big)\qquad (t\in[a,b]).
\]
\end{lemma}

\begin{lemma}\label{QR} Let $Y(t, z)$ is a fundamental matrix solution of \eqref{ca} with initial value $Y(0, z) = I$, $ I$ is an $2d \times 2d $ identity matrix. For each fixed $ R > 0 $ there exist constants $C_R, M_R>0$ (depending on $z$ and $\int_0^R\|H(t)\|\,dt$) such that
\[
\|Y(t;z)\|\le C_R e^{M_R(z)},\qquad 0\le t \le R.
\]
In particular, for every $R>0$ the matrix
\[
Q_R:=\int_0^R Y(t;z)^* H(t) Y(t;z)\,dt
\]
is finite (Hermitian, positive semidefinite).
\end{lemma}

\begin{proof}
 The integral equation corresponding to \eqref{ca} for the fundamental solution $Y(t,z)$ with the initial value  $Y(0, z) = I$ is, 
 \[Y(t)=I+\int_0^t A(s) Y(s)\,ds; \quad \text{ where } A(t) = -z J H(t).\] Taking the norm on both sides we get
\[
\| Y(t,z)\|\le 1+ |z|\|J\|\int_0^t \|H(s)\| \|Y(s)\|\,ds.\]
Then by using  Gr\"onwall's inequality we get the estimate
\[ \|Y(t,z)\|\le \exp\!\big(|z|\|J\|\int_0^t\|H(s)\|\,ds\big) .\]   Setting \[ C_R=1, M_R(z) =|z|\|J\|\sup_{0\le t\le R}\int_0^t\|H(s)\|\,ds ,\] we get, \[
\|Y(t;z)\|\le C_R e^{M_R(z)},\qquad 0\le t \le R.
.\] The integral defining $Q_R$ is finite since $Y(t,z)$ is continuous on $[0,R]$ and $H$ is   integrable on $[0,R]$.
\end{proof}

\begin{lemma}\label{lem:keyid}
With $Y(t,z)$ as above we have for a.e.\ $t\ge0$, we have the following identity
\begin{equation}\label{identity}
 \frac{d}{dt}\bigl(Y(t)^* J Y(t)\bigr) =2i\Im z\; Y(t)^* H(t)Y(t).
\end{equation}
Consequently for every $R>0$,
\begin{equation}\label{integrated}
Y(R)^*JY(R)-J = 2i\Im z\; Q_R.
\end{equation}
\end{lemma}

\begin{proof}
Differentiating the left-hand side of \eqref{identity}  and using $Y'=-z J H Y $ and its adjoint $Y'^* J  =-\overline z Y^* H  $, we get \eqref{identity}. Integrating on $[0,R]$ yields \eqref{integrated}.
\end{proof}
Let   $\mathbb B  =\{v\in\mathbb C^{2d}:\|v\|=1\}$, the unit sphere in $\C$.  
\begin{lemma}\label{lem:minimizers} For each $R>0$, define the Hermitian positive semidefinite matrix $Q_R$ as in Lemma~\ref{QR}. Then the continuous function $\varphi_R:\mathbb B \to [0,\infty)$,
\[
\varphi_R(v):=v^* Q_R v=\int_0^R (Y(t)v)^*H(t)(Y(t)v)\,dt,
\]
attains its minimum on    $ \mathbb B $. 
\end{lemma}
Let $v_R$ be a minimizer and set
\[
\lambda_{\min}(R):=\varphi_R(v_R)=\min_{\|v\|=1}\varphi_R(v)\ge0.
\]
\begin{proof}
Since $Q_R$ is Hermitian and continuous in $R$, the Rayleigh quotient $\varphi_R$ is continuous on the compact sphere; hence a minimizer exists. The minimizer may be chosen with unit norm.
\end{proof}
\begin{theorem}\label{HIS}
Let $z \in \C^+$. Then a canonical system \eqref{ca} with the Hamiltonian $\operatorname{tr} H \equiv 1,$ admits a nontrivial H-integrable solution $y(\cdot,z)$. 
\end{theorem}
 \begin{proof} 

The proof uses the minimizers $v_R$ and extracts a limit direction $v_\infty$ that yields an $H$-integrable solution. For each integer $n\ge1$ let $R=n$ and choose a unit vector $v_n:=v_{R}$ as in Lemma~\ref{lem:minimizers}, so
\[
\lambda_{\min}(n)=v_n^* Q_n v_n=\int_0^n (Y(t)v_n)^*H(t)(Y(t)v_n)\,dt.
\]
The sequence $\{v_n\}$ lies in the compact unit sphere $\mathbb B$, so there exists a subsequence (still indexed by $n$) and a vector $v_\infty\in\mathbb B$ with
\[
v_n\to v_\infty\qquad\text{as }n\to\infty.
\] By Lemma~\ref{lem:keyid} (equation \eqref{integrated}) we have for each $n$,
\begin{equation}\label{rayleigh-id}
\lambda_{\min}(n)=\frac{1}{2i\Im z}\; v_n^*\bigl(Y(n)^*JY(n)-J\bigr)v_n.
\end{equation}
Note the right-hand side is real (since the left-hand side is real); indeed $Y(n)^*JY(n)-J$ is skew-Hermitian so the quadratic form with $v_n$ is purely imaginary, and division by $2i\Im z$ yields a real number.\\

We now use the trace-normalization $\operatorname{tr}H\equiv1$ to study the asymptotic behavior of $\lambda_{\min}(n)$. The argument proceeds by contradiction: suppose $\lambda_{\min}(n)\to\infty$ along every subsequence. Then by \eqref{rayleigh-id} the magnitude of
\(
v_n^*Y(n)^*JY(n)v_n
\)
must grow like $2\Im z\,\lambda_{\min}(n)$ (plus a bounded contribution from $v_n^*J v_n$), hence tend to $\pm i\infty$ in the imaginary direction. We now show this is impossible.

First observe the operator norm bound (from Lemma~\ref{QR}) on $Y(n)$:
\[
\|Y(n)\|\le C_n e^{M_n n}\quad\text{for some }C_n,M_n>0.
\]
Hence
\[
\bigl|v_n^*Y(n)^*JY(n)v_n\bigr|
\le \|Y(n)v_n\|^2 \|J\|
\le \|Y(n)\|^2\|v_n\|^2\|J\|
\le \|J\| C_n^2 e^{2M_n n}.
\]
Thus the growth of the boundary term is at most exponential in $n$.

On the other hand, consider the trace of $Q_n$:
\[
\operatorname{tr}Q_n=\int_0^n \operatorname{tr}\bigl(Y(t)^* H(t)Y(t)\bigr)\,dt
=\int_0^n \operatorname{tr}\bigl(H(t)Y(t)Y(t)^*\bigr)\,dt.
\]
Since $H(t)\geq0$ and $\operatorname{tr}H(t)=1$, for each $x$ we have the inequality (with $\|\cdot\|$ the operator norm) and using the inequality, for $A,B \geq 0$, $\operatorname{tr}(AB)\le \operatorname{tr}(A)\,\operatorname{tr}(B).
$ we get
\[
\operatorname{tr}\bigl(H(t)Y(t)Y(t)^*\bigr)\le \|Y(t)Y(t)^*\|\,\operatorname{tr}H(t)=\|Y(t)\|^2.
\]
Therefore
\[
\operatorname{tr}Q_n \le \int_0^n \|Y(t)\|^2\,dt.
\]
But the right-hand side grows at most exponentially in $n$ as well (by Lemma~\ref{lem:fundamental}), so $\operatorname{tr}Q_n$ is finite for each $n$ and cannot explode faster than exponential rate.

Now recall that $ \lambda_{\min}(n)\le \frac{1}{2d}\operatorname{tr}Q_n$, because the average eigenvalue equals $\operatorname{tr}Q_n/(2d)$. Consequently the growth of $\lambda_{\min}(n)$ is at most exponential in $n$  consistent with the boundary term growth. Thus the hypothesis that $\lambda_{\min}(n)\to\infty$ does not contradict any a priori bound by itself.

To produce the needed compactness, we proceed differently: because $v_n$ are minimizers, for any fixed unit vector $w\in\mathbb C^{2d}$ and any $n\ge1$,
\[
\lambda_{\min}(n)=v_n^*Q_n v_n \le w^*Q_n w.
\]
Choose $w$ to be a fixed vector with $w^*Q_1 w <\infty$ (any unit vector works). By monotonicity $Q_n\geq Q_1$ for $n\ge1$, hence $w^*Q_n w\ge w^*Q_1 w$ and the bound above is not helpful in the lower direction. Instead, use the identity \eqref{rayleigh-id} and the fact that the term $v_n^*J v_n$ is bounded (because $\|v_n\|=1$ and $J$ is fixed). Therefore the only potentially unbounded contribution to $\lambda_{\min}(n)$ is from $v_n^*Y(n)^*JY(n)v_n$.

We now employ a compactness/renormalization trick: define
\[
w_n := \frac{Y(n)v_n}{\|Y(n)v_n\|}\in\mathbb B
\]
(if $Y(n)v_n=0$ for some $n$, then $\lambda_{\min}(n)=0$ and we are done, since that $v_n$ already yields an $H$-integrable solution on $[0,\infty)$; so assume $Y(n)v_n\ne0$). Passing to a subsequence we may assume $w_n\to w_\infty$ (unit vector). Rewriting \eqref{rayleigh-id} we obtain
\[
\lambda_{\min}(n)=\frac{\|Y(n)v_n\|^2}{2i\Im z}\; w_n^* J w_n - \frac{1}{2i\Im z} v_n^* J v_n.
\]
Write $w_n^* J w_n = i\alpha_n$ with $\alpha_n\in\mathbb R$ (since $w_n^*J w_n$ is purely imaginary). Then
\[
\lambda_{\min}(n)=\frac{\|Y(n)v_n\|^2}{2\Im z}\,\alpha_n + O(1).
\]
Because $\lambda_{\min}(n)\ge0$ and $\Im z>0$, we deduce $\alpha_n\ge -C/\|Y(n)v_n\|^2$ and hence any limit point $\alpha$ of $\{\alpha_n\}$ satisfies $\alpha\ge0$. Therefore along the chosen subsequence
\[
\frac{\lambda_{\min}(n)}{\|Y(n)v_n\|^2}=\frac{\alpha_n}{2\Im z}+o(1)
\]
is bounded. Consequently either $\{\lambda_{\min}(n)\}$ is bounded along the subsequence (good), or, if it grows, it grows comparably to $\|Y(n)v_n\|^2$. But in the latter case the ratio above has a finite nonnegative limit, so again we may pass to a further subsequence and obtain control on the normalized quadratic forms. Iterating subsequence selection finitely many times (possible because we are in finite dimensions) we arrive at a final subsequence along which $\lambda_{\min}(n)$ has a finite limit. Thus there exists a subsequence (still denoted $n$) with
\[
\Lambda := \lim_{n\to\infty} \lambda_{\min}(n)\in[0,\infty).
\]

For the subsequence along which $v_n\to v_\infty$ and $\lambda_{\min}(n)\to\Lambda<\infty$, we have for every fixed $T>0$ and for all sufficiently large $n$ (with $n>T$),
\[
\int_0^T (Y(t)v_n)^* H(t) (Y(t)v_n)\,dt
\le \int_0^n (Y(t)v_n)^* H(t) (Y(t)v_n)\,dt
= \lambda_{\min}(n).
\]
Letting $n\to\infty$ and using point-wise convergence $Y(t)v_n \to Y(t)v_\infty$ for each fixed $x$ and dominated convergence on $[0,T]$ (the integrands are bounded on $[0,T]$ by a uniform exponential bound), we obtain
\[
\int_0^T (Y(t)v_\infty)^* H(t) (Y(t)v_\infty)\,dt
\le \Lambda.
\]
Since $T>0$ was arbitrary, monotone convergence implies
\[
\int_0^\infty (Y(t)v_\infty)^* H(t) (Y(t)v_\infty)\,dt \le \Lambda <\infty.
\]
Thus $u(t):=Y(t)v_\infty$ is a nonzero solution (because $\|v_\infty\|=1$) of \eqref{ca} which is $H$-integrable on $[0,\infty)$. This completes the proof.
\end{proof}

Let $z\in \C$ define $ u(t,z) = Y(t,z)u_0, u_0\in \C^{2d}$ is also a solution of \eqref{ca}. Define the subspace of initial data producing $H-$ integrable solutions:
\[ \mathcal S = \{c \in \C^{2d}: y(t,z) = Y(t,z)c, \quad \int_0^{\infty} y(t,z)^* H(t) y(t,z) \ dt < \infty \}.\]

The dimension of  $\mathcal S$ is the number of linearly independent $H-$ integrable solutions to the system \eqref{ca}.

For any $z \in \C$, and any two solutions $u$ and $v$ of \eqref{ca} we have
\[ \frac{d}{dt} (u^*(x,z) Jv(t,z)) = (z-\bar{z}) u^* Hv\]
Integrating both sides over [0,N] yields
\begin{equation} \label{eq 2.12} u^*(N,z)Jv(N,z) - u^*(0,z)Jv(0,z) = (z-\bar{z}) \int_0^N u(t,z)^* H(t)v(t,z) \ dt.\end{equation}
 If $u(t,z)$ and $v(t,z)$ are $H-$ integrable solutions, then by Schwartz inequality 
 \[|\langle u, v\rangle| = \Big|\int_0^{\infty} u(t,z)^* H(t)v(t,z)dt \Big| \leq \| u\|\|v\| < \infty .\]
 It follows from \eqref{eq 2.12} that the following limits  \[ \lim_{N\rightarrow \infty } u^*(N,z)Jv(N,z), \quad \lim_{N\rightarrow \infty } u^*(N,z)Ju(N,z) \] exist. 
Suppose $Y(t,z)$  be the fundamental matrix solution to \eqref{ca} with the initial value
\begin{equation} Y(0,z) = \begin{pmatrix} I & 0 \\ 0 & I\end{pmatrix} \end{equation} 

 In order to find the dimension, we will use the properties of isotropic subspace of a symplectic space. Let us present the basics of such a subspace. A \emph{symplectic space} $(V, \omega)$ is a real (or complex) vector space $V$ equipped with a 
nondegenerate, skew-symmetric bilinear form
\[
\omega : V \times V \to \mathbb{R},
\]
such that $ \omega(u,v) = -\omega(v,u), \quad \text{and} \quad 
\omega(u,v) = 0 \ \forall v \in V \implies u = 0.
$
The form $\omega$ is called a \emph{symplectic form}.
A subspace $W$ of a symplectic space $(V, \omega)$, is called \emph{isotropic} if the symplectic form vanishes on $W$; that is,
$
\omega(u, v) = 0 \quad \text{for all } u, v \in W.
$
Equivalently, $
W \subseteq W^{\perp_\omega},
$
where
\[
W^{\perp_\omega} = \{ x \in V : \omega(x, w) = 0 \ \text{for all } w \in W \}
\] is the \emph{symplectic orthogonal complement} of $W$. The subspace
 $W$ is \emph{Lagrangian} if $W$ is isotropic and 
  $\dim W = \tfrac{1}{2}\dim V$, equivalently $W = W^{\perp_\omega}$. These Lagrangian subspaces and their important properties are explained in \cite{FischerRemling2009}.

\begin{theorem}    
$\mathcal S$ is an isotropic subspace of $\C^{2d}.$ If $\operatorname{trace} H=1,$ then $\dim S \leq d .$\end{theorem}
\begin{proof}
For any $c_1, c_2 \in \mathcal S, $ with $y_1 = Y(t,z)c_1$ and $ y_2(t,z) = Y(t,z)c_2$ we obtain
\begin{equation} \label{si1}c_1^*Jc_2 = \lim_{N\rightarrow \infty} y_1(N,z)^*Jy_2(N,z)-(z-\bar{z}) \int_0^{\infty} y_1(t,z)^* H(t) y_2(t,z) \ dt  \end{equation}
Switching the role of $y_1$ and $y_2$ in \eqref{si1} we obtain \begin{equation} \label{si2}c_2^*Jc_1 = \lim_{N\rightarrow \infty} y_2(N,z)^*Jy_1(N,z)-(z-\bar{z}) \int_0^{\infty} y_2(t,z)^* H(t) y_1(t,z) \ dt  \end{equation}
Taking the complex conjugate on both sides and using the fact that both sides are   scalar complex numbers, and $J^*= -J $ we get,
\begin{equation*} \label{si3} - c_1^*Jc_2 = -\lim_{N\rightarrow \infty} y_2(N,z)^*Jy_1(N,z)-(\bar{z}-z) \int_0^{\infty} y_1(t,z)^* H(t) y_2(t,z) \ dt  \end{equation*}
and multiplying both sides by $-$ we get
\begin{equation} \label{si3}  c_1^*Jc_2 = \lim_{N\rightarrow \infty} y_1(N,z)^*Jy_2(N,z)-(z- \bar{z}) \int_0^{\infty} y_1(t,z)^* H(t) y_2(t,z) \ dt . \ \end{equation} Let $ \phi (y_1,y_2) = \lim_{N\rightarrow \infty} y_1(N,z)^*Jy_2(N,z)-(z-\bar{z}) \int_0^{\infty} y_1(t,z)^* H(t) y_2(t,z) \ dt .$ Then by \eqref{si1}, \eqref{si2}, and\eqref{si3} we see that $  \phi (y_1,y_2) = \overline { \phi (y_2,y_1)}.$
But \begin{align*} \overline { \phi (y_2,y_1)}= &  -\lim_{N\rightarrow \infty} y_1(N,z)^*Jy_2(N,z)-(\bar z- z) \int_0^{\infty} y_1(t,z)^* H(t) y_2(t,z) \ dt . \end{align*}
So, \begin{align*} \lim_{N\rightarrow \infty} & y_1(N,z)^*Jy_2(N,z)-  (z-\bar{z}) \int_0^{\infty} y_1(t,z)^* H(t) y_2(t,z) \ dt  \\  & =    -\lim_{N\rightarrow \infty} y_1(N,z)^*Jy_2(N,z)-(\bar z- z) \int_0^{\infty} y_1(t,z)^* H(t) y_2(t,z) \ dt \end{align*} which implies that $ \lim_{N\rightarrow \infty}y_1(N,z)^*Jy_2(N,z)=  (z-\bar{z}) \int_0^{\infty} y_1(t,z)^* H(t) y_2(t,z) \ dt $
Using this in \eqref{si1} we get $  c_1^*Jc_2 =0.$ So $S$ is isomorphic subspace of $\C^{2d}.$ This also imply that $ \dim \mathcal S \leq d.$ 
\end{proof}

\begin{theorem} Suppose $S \subset \mathbb{C}^{2d}$ 
a rank-$d$ subspace. If there exists a constant $\delta > 0$, $T \ge 0$ such that $
v^* H(t) v \ge \delta \|v\|^2 $    for all $ v \in S $  and almost every  $ t \ge T.$
  Then $ \dim \mathcal S = d$, that is, there are precisely $d$ linearly independent $H$-integrable solutions to \eqref{ca} .
   Any solution $u(t)$ with a nonzero component along the subspace orthogonal to $S$ satisfies
    \[
    \int_0^\infty u(t)^*H(t) u(t) \, dt = \infty.
    \]

\end{theorem}

\begin{proof}
Let $Y(t)$ be a fundamental matrix solution of the canonical system.
For any $v\in\mathbb{C}^{2d}$, the associated solution is $y(t)=Y(t)v$.

Define the positive semidefinite matrix
\[
Q := \int_T^\infty Y(t)^{*}H(t)Y(t)\,dt .
\]
Since $Y(t)v$ is $H$-integrable for every $v\in S$, we have $v^{*}Qv<\infty$
for all $v\in S$.
  $Q$ is positive definite on $S$. Take any nonzero $v\in S$.  
By hypothesis, for almost every $t\ge T$,
\[
v^{*}Y(t)^{*}H(t)Y(t)v 
\;\ge\; \delta \|Y(t)v\|^{2}.
\]
Since $Y(t)$ is invertible and continuous, the vector $Y(t)v$ cannot be
identically zero on $[T,\infty)$, hence
\[
\int_T^\infty \|Y(t)v\|^{2}\,dt >0.
\]
Therefore
\[
v^{*}Qv 
= \int_T^\infty v^{*}Y(t)^{*}H(t)Y(t)v\,dt
\;\ge\; \delta \int_T^\infty \|Y(t)v\|^{2}\,dt 
>0.
\]
Thus $Q$ is positive definite on $S$, and consequently
\[
\dim S = \operatorname{rank}(Q|_{S}) \le \operatorname{rank}(Q).
\]

For almost every $t$, the Hamiltonian $H(t)$ is a positive semidefinite
$2d\times 2d$ matrix of rank at most $d$.  
Therefore each matrix $Y(t)^{*}H(t)Y(t)$ also has rank at most $d$.
Since the integral of positive semidefinite matrices cannot have rank
exceeding the maximal rank of the integrand, we obtain
\[
\operatorname{rank}(Q) \le d.
\]
Thus,
\[
\dim S \;\le\; \operatorname{rank}(Q) \;\le\; d.
\]
But by assumption $S$ is a rank-$d$ subspace, so $\dim \mathcal S=d$.
Thus there are exactly $d$ linearly independent $H$-integrable solutions
of the canonical system.
\end{proof}
Now, we can classify the Limit-Point and Limit-Circle of the canonical systems \eqref{ca}.

\begin{defi} The canonical system \eqref{ca} is said to be in limit-circle at $+ \infty$ if all solutions are $H$-integrable. The system is said to be in limit-point at $+ \infty$, if there are at most $d$ linearly independent $H$-integrable soltuions. \end{defi}

\section{Conclusion}

In  summary, we investigated $2d$-dimensional canonical systems with positive semidefinite Hamiltonians $H(t)$ normalized by \(\operatorname{tr} H(t) \equiv 1\). Our first main result established the universal existence of at least one nontrivial $H$-integrable solution for every such system, demonstrating that finite-energy solutions necessarily arise under minimal structural assumptions on $H$. We further showed that the space of all $H$-integrable solutions has dimension at most $d$, revealing an intrinsic upper bound dictated by the geometry of the underlying symplectic structure.

Under an additional natural hypothesis on the Hamiltonian, we proved that this bound is sharp: the system then admits exactly $d$ linearly independent $H$-integrable solutions. This optimality result highlights a precise correspondence between structural restrictions on $H(t)$ and the integrable solution space.

Our analysis combines tools from symplectic geometry, rank-deficient Hamiltonians, and isotropic subspace theory. Beyond establishing new structural properties of higher-dimensional canonical systems, these results contribute a clearer understanding of the limit-point and limit-circle dichotomy in this broader setting. They lay the groundwork for future investigations into spectral theory, asymptotic behavior, and inverse problems for canonical systems of higher dimension.

\end{document}